\documentclass{amsart}
\usepackage{amssymb}
\def\R{{\mathbb R}}

\def\<{\langle}
\def\>{\rangle}
\def\P{\mathbb P}

\def\E{\mathbb E}

\def\0{\underline 0}

\def\e{\underline e}

\def\1{\underline 1}

\newcommand{\bel}{\begin{equation}\label}
\newcommand{\ee}{\end{equation}}
      \newtheorem{theorem}{Theorem}[section]

       \newtheorem{remark}{Remark}[section]
\theoremstyle{definition}

\title[Regression characterization of the gamma and Kummer laws]{Regression version of the Matsumoto-Yor type characterization of the Gamma and Kummer distributions}
\author{Jacek Weso\l owski}
\address{Wydzia{\l} Matematyki i Nauk Informacyjnych, Politechnika Warszawska, Warszawa, Poland}
\email{wesolo\@ mini.pw.edu,pl}
\date{\today}
\begin{document}
\begin{abstract}
In this paper we study a Matsumoto-Yor type property for the gamma and Kummer independent variables discovered in Koudou and Vallois (2012). We prove that constancy of regressions of $U=(1+(X+Y)^{-1})/(1+X^{-1})$ given $V=X+Y$ and of $U^{-1}$ given $V$, where $X$ and $Y$ are independent and positive random variables, characterizes the gamma and Kummer distributions. This result completes characterizations by independence of $U$ and $V$ obtained, under smoothness assumptions for densities, in Koudou and Vallois (2011, 2012). Since we work with differential equations for the Laplace transforms, no density assumptions are needed.
\end{abstract}

\maketitle
\section{Introduction}
Let $X$ and $Y$ be independent random variables. There are several well known settings in which $U=\psi(X,Y)$ and $V=X+Y$ are also independent. Related characterizations of distributions of $X$ and $Y$ by properties of independence of $X$ and $Y$ and independence of $U$ and $V$ have been also studied. The most prominent seem to be: \begin{itemize}
\item Bernstein (1941) characterization of the normal law by independence of $U=X-Y$ and $V$,
\item Lukacs (1956) characterization of the gamma law by independence of $U=X/Y$ and $V$.
\end{itemize}

In the end of 1990's  a new independence phenomenon of this kind, called Matsumoto-Yor property, see e.g. Stirzaker (2005), p. 43,  was discovered. It says that for $X$ with a GIG (generalized inverse Gaussian) law and independent $Y$ with a gamma law (both distributions with suitably adjusted parameters), random variables $U=1/X-1/(X+Y)$ and $V$ are independent. This elementary property was identified while the authors analyzed structure of functionals of  Brownian motion - see Matsumoto and Yor (2001, 2003). A related characterization of the GIG and gamma laws by independence of $X$ and $Y$ and of $U$ and $V$ was obtained in Letac and Weso\l owski (2000). Both these results: the Matsumoto-Yor property and the characterization were generalized in several directions. Matrix variate analogues were studied e.g. in Letac and Weso\l owski (2000), Weso\l owski (2002) and Massam and Weso\l owski (2006) - the last one including a relation with conditional structure of Wishart matrices. Recently it has been extended to symmetric cones setting in Ko\l odziejek (2014). Multivariate versions related to specific transformations governed by a tree were considered in Barndorff-Nielsen and Koudou (1998), Massam and Weso\l owski (2004) and  Koudou (2006) . Further connections with (exponential) Brownian motion were investigated in Weso\l owski and Witkowski (2007) and Matsumoto et al. (2009). There are also regression versions of Matsumoto-Yor typ characterizations, as given in Seshadri and Weso\l owski (2001),  Weso\l owski (2002) and Chou and Huang (2004). A survey of these results together with other characterizations of the GIG law can be found in a recent paper Koudou and Ley (2014).

In 2009 Koudou and Vallois tried to generalize Matsumoto-Yor property by a search of distributions of independent $X$ and $Y$ and functions $f$ such that $V=f(X+Y)$ and $U=f(X)-f(X+Y)$ are independent. Their research lead to a discovery of another pair $U=\psi(X,Y)$ and $V=X+Y$ with independence property:
Assume that $X$ and $Y$ are independent random variables, $X$ has the Kummer distribution $\mathrm{K}(a,b,c)$ with the density
$$
f_X(x)\propto \tfrac{x^{a-1}e^{-cx}}{(1+x)^{a+b}}\,I_{(0,\infty)}(x), \qquad a,b,c>0,
$$
and $Y$ has the gamma distribution $\mathrm{G}(b,c)$ with the density
$$
f_Y(y)\propto y^{b-1}e^{-cy}I_{(0,\infty)}(y).
$$
Then, see Koudou and Vallois (2012), random variables
\bel{juvi}
U=\tfrac{1+\tfrac{1}{X+Y}}{1+\tfrac{1}{X}}\qquad \mbox{and}\qquad V=X+Y
\ee
are independent, $U$ has the beta first kind distribution $\mathrm{B}_I(a,b)$ with the density
$$
f_U(u)\propto u^{a-1}(1-u)^{b-1}I_{(0,1)}(u)
$$
and $V$ has the Kummer distribution, $\mathrm{K}(a+b,-b,c)$. (Note that the Kummer distribution $K(\alpha,\beta,\gamma)$ is well-defined iff $\alpha,\gamma>0$ and $\beta\in\R$.)

It is an interesting question if a theory, similar to the one for the original Matsumoto-Yor property described in the literature recalled above, can be developed for this new independence property. There have already been some successful efforts in this direction. The property was extended to matrix variate distributions in Koudou (2012). It is also known, see Koudou and Vallois (2011, 2012), that, under appropriate smoothness assumptions on densities, a characterization counterpart of the property holds: if $X$ and $Y$ are independent positive random variables, and $U$ and $V$, given by \eqref{juvi}, are also independent then $X\sim \mathrm{K}(a,b,c)$ and $Y\sim \mathrm{G}(b,c)$ for some positive constants $a,b,c$. These smoothness restrictions require that the densities of $X$ and $Y$ are strictly positive on $(0,\infty)$ and either twice differentiable or their logarithms are locally integrable.  Letac (2009) conjectured that such a characterization is possibly true with no assumptions on densities.  In this note we contribute further to this development following the characterization  path. Our aim is to show a characterization of the gamma and Kummer laws without referring to densities at all. Actually, we will consider constancy of regressions condition which, up to necessary moment assumption, is weaker than independence.

\section{Regression characterization}
Our main result is a characterization of the Kummer and gamma laws by constancy of regressions of $U$ and $U^{-1}$ given $V$ in the setting described in \eqref{juvi}. Since $U\in(0,1)$ $\P$-a.s. $\E\,U<\infty$, and one can consider conditional moment $\E(U|V)$ without any additional restrictions. This is not the case of $\E(U^{-1}|V)$ since, a priori, the moment $\E\,U^{-1}$ may not be finite.  Since
$$
U^{-1}=\tfrac{1+X}{X}\,\tfrac{X+Y}{1+X+Y}\le 1+\tfrac{1}{X}
$$
we have $\E\,U^{-1}\le 1+\E\,X^{-1}$. So, under the assumption $\E\,X^{-1}<\infty$ the conditional moment $\E(U^{-1}|V)$ is well defined.

Now we are ready to state the main result of this note.

\begin{theorem}\label{REG KV}
Let $X$ and $Y$ be independent positive non-degenerate random variables and $\E\,X^{-1}<\infty$. Define $U$ and $V$ through \eqref{juvi}. If
\bel{reg}
\E(U|V)=\alpha \qquad \mbox{and}\qquad \E(U^{-1}|V)=\beta
\ee
for real constants $\alpha$ and $\beta$ then there exists a constant $c>0$ such that $$X\sim \mathrm{K}\left(1+\tfrac{1-\alpha}{\alpha\beta},\tfrac{(1-\alpha)(\beta-1)}{\alpha\beta},c\right)\qquad\mbox{and}\qquad Y\sim\mathrm{G}\left(\tfrac{(1-\alpha)(\beta-1)}{\alpha\beta},\,c\right).$$
\end{theorem}

\begin{proof}
First, rewrite the equations \eqref{reg} as
$$
\E\left(\left.\tfrac{X}{1+X}\right|X+Y\right)=\alpha\tfrac{X+Y}{1+X+Y}\qquad\mbox{and}\qquad
\E\left(\left.\tfrac{1+X}{X}\right|X+Y\right)=\beta\tfrac{1+X+Y}{X+Y}.
$$
Equivalently, we have
\bel{reg1}
\E\left(\left.\tfrac{1}{1+X}\right|X+Y\right)=1-\alpha+\tfrac{\alpha}{1+X+Y}
\ee
and
\bel{reg2}
\E\left(\left.\tfrac{1}{X}\right|X+Y\right)=\beta-1+\tfrac{\beta}{X+Y}.
\ee
The equation \eqref{reg1} implies
\bel{eq1}
\E\,\tfrac{e^{s(1+X+Y)}}{1+X}=(1-\alpha)\E\,e^{s(1+X+Y)}+\alpha\E\,\tfrac{e^{s(1+X+Y)}}{1+X+Y}
\ee
at least for $s\le 0$.

Similarly, from \eqref{reg2} we get the equation
\bel{eq2}
\E\,\tfrac{e^{s(X+Y)}}{X}=(\beta-1)\E\,e^{s(X+Y)}+\beta\E\,\tfrac{e^{s(X+Y)}}{X+Y},\quad s\le 0.
\ee

Differentiating \eqref{eq1} with respect to $s$ (it is possible at least for $s<0$) we obtain
$$
\E\,e^{s(1+X+Y)}+\E\,\tfrac{Y}{1+X}\,e^{s(X+Y+1)}=(1-\alpha)\E\,(1+X+Y)e^{s(1+X+Y)}+\alpha\,\E\,e^{s(X+Y+1)}.
$$
After dividing by $e^s$  both sides of the above equation and canceling the term $\E\,e^{s(X+Y)}$ we arrive at
$$
\E\,\tfrac{e^{sX}}{1+X}\,\E\,Ye^{sY}=(1-\alpha)\left(\E\,Xe^{sX}\,\E\,e^{sY}+\E\,e^{sX}\,\E\,Ye^{sY}\right).
$$
This equation can be written as
\bel{EQ1}
e^{-s}K\,M'=(1-\alpha)(L\,M)',
\ee
where
$$K(s)=\E\,\tfrac{e^{s(1+X)}}{1+X},\qquad L(s)=\E\,e^{sX}\qquad \mbox{and}\qquad M(s)=\E\,e^{sY}.$$

On the other hand differentiating \eqref{eq2} we get
$$
\E\,e^{s(X+Y)}+\E\,\tfrac{Y}{X}\,e^{s(X+Y)}=(\beta-1)\E\,(X+Y)e^{s(X+Y)}+\beta\,\E\,e^{s(X+Y)}.
$$
Consequently,
$$
\E\,\tfrac{e^{sX}}{X}\,\E\,Y\e^{sY}=(\beta-1)\left(\E\,Xe^{sX}\,\E\,e^{sY}+\E\,e^{sX}\,\E\,Ye^{sY}+\E\,e^{sX}\,\E\,e^{sY}\right).
$$
Therefore
\bel{EQ2}
G\,M'=(\beta-1)((L\,M)'+L\,M),
\ee
where
$$
G(s)=\E\,\tfrac{e^{sX}}{X}.
$$

By deriving the formula for $(LM)'$ from  \eqref{EQ1} and \eqref{EQ2} get
\bel{EQ}
ae^{-s}\,K\,M'=bG\,M'-L\,M, \ee
with $a=(1-\alpha)^{-1}$ and $b=(\beta-1)^{-1}$. Differentiate \eqref{EQ} to get
$$
-ae^{-s}KM'+ae^{-s}K'M'+ae^{-s}KM''=bG'M'+bGM''-(LM)'.
$$
Note that $G'=L=e^{-s}K'$. Therefore the above equation together with \eqref{EQ1} and \eqref{EQ2}, after multiplying both sides by $M'$ implies
$$
-(LM)'M'+aLM'^2+(LM)'M''=bLM'^2+((LM)'+LM)M''-(LM)'M'
$$
which after cancelations (which are allowed in a left neighborhood of zero)  gives
$$
MM''=(a-b)M'^2.
$$
Note that
$$
a-b=\tfrac{1}{1-\alpha}+\tfrac{1}{1-\beta}=\tfrac{2-\alpha-\beta}{(1-\alpha)(1-\beta)}=1+\tfrac{\alpha\beta-1}{(1-\alpha)(\beta-1)}=:1+\tfrac{1}{p}
$$
and due to obvious inequalities: $\alpha<1$, $\beta>1$ and $\alpha\beta>1$, we conclude that  $p>0$. Consequently, $Y$ has a gamma distribution $\mathrm{G}(p,\,c)$, where $c$ is a positive constant.

Now we differentiate equation \eqref{EQ2} for $s<0$ getting
$$
b G'M'+bGM''=(LM)''+(LM)'.
$$
Multiplying both sides by $M'$ and using again \eqref{EQ2} we arrive at
$$
bLM'^2+((LM)'+LM)M''=(LM)''M'+(LM)'M'
$$
which yields
$$
L''\tfrac{M'}{M}+L'\left[2\left(\tfrac{M'}{M}\right)^2+\tfrac{M'}{M}-\tfrac{M''}{M}\right]
-L\left[\tfrac{M''}{M}-(1-b)\left(\tfrac{M'}{M}\right)^2\right]=0.
$$
After inserting known values for $M$, $M'$ and $M''$ the above equation transforms into
$$
(c-s)L''(s)+(1-p+c-s)L'(s)-(1+bp)L(s)=0,\qquad s\le 0.
$$
Change the variable $t:=c-s$ and define $N(t)=L(c-t)$. It follows that
$$
tN''(t)+(1-p-t)N'(t)-(1+bp)N(t)=0\qquad t\ge c.
$$
We read two linearly independent solutions of this equation from Abramovitz and Stegun (1965), Ch. 13. One of these solutions is the generalized hypergeometric function $$N(t)=M(1+bp,1-p,t)=_1\hspace{-1mm}F_1(1+bp,1-p,t)$$ which is of the order $e^tt^{(1+p)b}$ for $t\to \infty$, see (13.1.4) in Abramovitz and Stegun (1965), and thus tends to infinity with $s\to-\infty$ and thus $t=c-s\to \infty$. In the case we consider this is impossible since the Laplace transform of negative argument $s$ of positive probability measure has to be bounded. The second solution
$$
N(t)=U(1+bp,1-p,t)=C\int_0^{\infty}\,e^{-tx}\tfrac{x^{bp}}{(1+x)^{p(b+1)+1}}\;dx,
$$
yields
$$
L(s)=C\int_0^{\infty}\,e^{sx}\tfrac{x^{bp}}{(1+x)^{p(b+1)+1}}e^{-cx}\;dx,
$$
which is a Laplace transform of the Kummer $\mathrm{K}(bp+1,p,c)$ distribution.
\end{proof}
\begin{remark}
Recall that $U\sim\mathrm{Beta}_I(1+bp,p)$ and thus $\E\,X^{-1}<\infty$, as expected. Moreover, since $p>1$ then also $\E\,(1-U)^{-1}<\infty$.
\end{remark}
\begin{remark}
It still not clear if independence of $U$ and $V$ for independent, positive and non-degenerate $X$ and $Y$ without any additional assumptions  characterizes the gamma and Kummer laws. Theorem 1 answers the question under additional restriction that $\E\,U^{-1}<\infty$.
\end{remark}
\begin{remark}
Since $U$ as defined in \eqref{juvi} is $(0,1)$ valued random variable, without any additional moment assumptions we can write regressions conditions of the form
$$
\E\left(\left.(1-U)^k\right|V\right)=\alpha_k
$$
for some positive $k$'s and $\alpha_k$'s. Obviously, such conditions are weaker than independence. A little of algebra allows to see that the above condition is equivalent to
$$
\E\left(\left.\tfrac{Y^k}{(1+X)^k}\right|X+Y\right)=\alpha_k(X+Y)^k.
$$
However, we failed to prove characterization assuming the above conditions for, say, $k=1,2$.
\end{remark}

\vspace{5mm}\noindent\small {\bf Acknowledgement.} This research has been partially supported by NCN grant No. 2012/05/B/ST1/00554. I am grateful to G. Letac for sending me his unpublished paper on the Kummer distribution.

\vspace{5mm}\noindent
{\bf References}

\begin{enumerate}

\item {\sc Abramovitz, M., Stegun, I.A.} {\em Handbook of Mathematical Functions with Formulas, Graphs and Mathematical Tables}, National Bureau of Standards, Applied Mathematics Series {\bf 55}, Washington, 1964.

\item {\sc Barndorff-Nielsen, O.E., Koudou, A.E.} Trees with random conductivities and the (reciprocal) inverse Gaussian distribution. {\em Adv. Appl. Probab.} {\bf 30} (1998), 409-424.

\item {\sc Bernstein, S.N.} On a property which characterizes Gaussian distribution. {\em Zap. Leningrad Polytech. Inst.} {\bf 217(3)} (1941), 21-22 (in Russian).

\item {\sc Chou, C.-W, Huang, W.-J.} On characterizations of the gamma and generalized inverse Gaussian distributions. {\em Statist. Probab. Lett.} {\bf 69} (2004), 381-388.

\item {\sc Ko\l odziejek, B.} The Matsumoto-Yor property and its converse on symmetric cones. {\em arXiv} {\bf 1409.5256} (2014), 1-10.

\item {\sc Koudou, A.E.} A link between the Matsumoto-Yor property and an independence property on trees. {\em Statist. Probab. Lett.} {\bf 76} (2006), 1097-1001.

\item {\sc Koudou, A.E.} A Matsumoto-Yor property for Kummer and Wishart matrices. {\em Statist. Probab. Lett.} {\bf 82(11)} (2012), 1903-1907.

\item {\sc Koudou, A.E., Ley, C.} Characterizations of GIG laws: a survey. {\em Probab. Surv.} {\bf 11} (2014), 161-176.

\item {\sc Koudou, A.E., Vallois, P.} Which dsitributions have the Matsumoto-Yor property? {\em Electr. Comm. Probab.} {\bf 16} (2011), 556-566.

\item {\sc Koudou, A.E., Vallois, P.} Independence properties of the Matsumoto-Yor type. {\em Bernoulli} {\bf 18(1)} (2012), 119-136.

\item {\sc Letac, G.} Kummer distributions. {\em Unpublished manuscript} (2009), 1-15.

\item {\sc Letac, G., Weso\l owski, J.} An independence property for the product of  GIG and gamma laws. {\em Ann. Probab.} {\bf 28} (2000), 1371-1383.

\item {\sc Lukacs, E.} A characterization of the gamma distribution. {\em Ann. Math. Statist.} {\bf 26} (1955), 319-324.

\item {\sc Massam, H., Weso\l owski, J.} The Matsumoto-Yor property on trees. {\em Bernoulli} {\bf 10} (2004), 685-700.

\item {\sc Massam, H., Weso\l owski, J.} The Matsumoto-Yor property and the structure of the Wishart distribution. {\em J. Mutivar. Anal.} {\bf 97} (2006), 103-123.

\item {\sc Matsumoto, H., Weso\l owski, J., Witkowski, P.} Tree structured independences for exponential Brownian functionals. {\em Stoch. Proc. Appl.} {\bf 119} (2009), 3798-3815.

\item {\sc Matsumoto, H., Yor, M.} An analogue of Pitman's $2M-X$ theorem for exponential Wiener functionals: Part II: The role of the generalized inverse Gaussian laws. {\em Nagoya Math. J.} {\bf 162} (2001), 65-86.

\item {\sc Matsumoto, H., Yor, M.}  Interpretation via Brownian motion of some independence properties between GIG and gamma variables. {\em Statist. Probab. Lett.} {\bf 61} (2003), 253-259.

\item {\sc Seshadri, V., Weso\l owski, J.} Mutual characterizations of the gamma and generalized inverse Gaussian laws by constancy of regression. {\em Sankhya, A} {\bf 63} (2001), 107-112.

\item {\sc Stirzaker, D.} {\em Stochastic Processes and Models}, Oxford Univ. Press, Oxford 2005.

\item {\sc Weso\l owski, J.} The Matsumoto-Yor independence property for GIG and gamma laws, revisited. {\em Math. Proc. Cambridge Philos. Soc.} {\bf 133} (2002), 153-161.

\item {\sc Witkowski, P., Weso\l owski, J.} Hitting times of Brownian motion and the Matsumoto-Yor property on trees. {\em Stoch. Proc. Appl.} {\bf 117} (2007), 1303-1315.

\end{enumerate}

\end{document}